\documentclass[reqno]{amsart}
\usepackage{amssymb}
\usepackage[hyphens]{url} 
\usepackage{mathtools}
\usepackage[active]{srcltx}
\usepackage[mathscr]{eucal}
\usepackage{xcolor}

\usepackage{amsmath,bm}
\usepackage{amsthm}
\usepackage{amsfonts}
\usepackage{verbatim}
\usepackage{latexsym}

\usepackage{color}
\usepackage{enumerate}

\usepackage[colorlinks=true]{hyperref}
\hypersetup{urlcolor=blue, citecolor=red}

\numberwithin{equation}{section}

\theoremstyle{plain}
 \newtheorem{thm}{Theorem}[section]
 \newtheorem{prop}{Proposition}[section]

\theoremstyle{definition}
 
 \newtheorem{dfn}{Definition}[section]
\theoremstyle{remark}
 \newtheorem{rem}{Remark}[section]
 \numberwithin{equation}{section}

\DeclarePairedDelimiterX{\expectarg}[1]{[}{]}{%
  \ifnum\currentgrouptype=16 \else\begingroup\fi
  \activatebar#1
  \ifnum\currentgrouptype=16 \else\endgroup\fi
}

\newcommand{\innermid}{\nonscript\;\delimsize\vert\nonscript\;}
\newcommand{\activatebar}{%
  \begingroup\lccode`\~=`\|
  \lowercase{\endgroup\let~}\innermid 
  \mathcode`|=\string"8000
}
\title[Data Assimilation for 3D NSE]{Continuous Data Assimilation for the Three Dimensional Navier-Stokes Equations }

\newcommand{\comments}[1]{}

\newcommand{\R}{\mathbb R}

\newcommand{\N}{\mathbb N}
\renewcommand{\frak}{\mathfrak}

\newcommand{\nn}{\nonumber}
\newcommand{\D}{\displaystyle }
\newcommand{\dt}{{\D\frac{d}{dt}}}
\newcommand{\ra}{\rightarrow}
\newcommand{\lra}{\longrightarrow}
\newcommand{\be}{\begin{equation}}
\newcommand{\ee}{\end{equation}}
\newcommand{\bes}{\begin{equation*}}
\newcommand{\ees}{\end{equation*}}

\newcommand{\bea}{\begin{eqnarray}}
\newcommand{\eea}{\end{eqnarray}}
\newcommand{\beas}{\begin{eqnarray*}}
\newcommand{\eeas}{\end{eqnarray*}}
\newcommand{\p}{\partial}

\def\mbb#1{\mathbb{#1}}
\def\cal#1{\mathcal{#1}}
\def\lf{\left} 
\def\rg{\right} 
\def\<{\langle} 
\def\>{\rangle}

\def\Om{\Omega}

\def\at#1{\Bigr|_{#1}}

\def\P{\mbb{P}_\sigma}

\def\tw{\widetilde{w}}

\def\co{{\cal O}}

\def\intav#1{\mathchoice
          {\mathop{\vrule width 6pt height 3 pt depth -2.5pt
                  \kern -9pt \intop}\nolimits_{\kern -6pt#1}}%
          {\mathop{\vrule width 5pt height 3 pt depth -2.6pt
                  \kern -6pt \intop}\nolimits_{#1}}%
          {\mathop{\vrule width 5pt height 3 pt depth -2.6pt
                  \kern -6pt \intop}\nolimits_{#1}}%
          {\mathop{\vrule width 5pt height 3 pt depth -2.6pt
                  \kern -6pt \intop}\nolimits_{#1}}}

\newcommand{\charfn}[1]{{\raisebox{1.2pt}{\mbox{$\chi
_{\kern-1pt\lower3pt\hbox{{$\scriptstyle{#1}$}}}$}}}}

\author[A. Biswas]{\bfseries
Animikh Biswas$^\dagger$}  
\author[R. Price]{\bfseries
Randy Price}
\address{ (Animikh Biswas) 
Department of Mathematics and Statistics\\ 
University of Maryland Baltimore County  \\ 
Baltimore, MD 21250\\
USA}
\email{abiswas@umbc.edu}
\address{ (Randy Price) 
Department of Mathematics and Statistics\\ 
University of Maryland Baltimore County  \\ 
Baltimore, MD 21250\\
USA}
\email{randyp1@umbc.edu}
\thanks{$^\dagger$ Corresponding author. Email: abiswas@umbc.edu}
\begin{document}
\comments{
{\begin{flushleft}
		\baselineskip9pt\scriptsize
MANUSCRIPT
\end{flushleft}}
}
\vspace{18mm} \setcounter{page}{1} \thispagestyle{empty}

\begin{abstract}
In this paper, we provide conditions, \emph{based solely on the observed data},  for the global well-posedness, regularity and convergence of the Azouni-Olson-Titi data assimilation algorithm (AOT algorithm) for  a Leray-Hopf weak solutions  of the three dimensional Navier-Stokes equations (3D NSE). 
The aforementioned conditions on the observations, which in this case comprise either of \emph{modal} or \emph{volume element observations}, are automatically satisfied for solutions that are globally regular and are uniformly bounded in the $H^1$-norm. However, neither regularity nor uniqueness is necessary for the efficacy of the AOT algorithm. To the best of our knowledge, this is the first such rigorous analysis of the AOT  data assimilation algorithm for the  3D NSE.
\end{abstract}

\maketitle

\section{Introduction}  

For a given dynamical system, which is believed to accurately describe some aspect(s) of an underlying physical reality, 
the problem of forecasting is often hindered by  inadequate knowledge of the initial state and/or model parameters describing the system. However, in many cases, this is compensated by the fact that one has access to data from (possibly noisy) measurements of the system, collected  on a {\it much coarser spatial grid than the desired resolution of the forecast.} The objective of data assimilation and signal synchronization is to use this coarse scale observational measurements to fine tune  our knowledge of the state and/or model to improve the accuracy of the forecasts \cite{Daley1991, Kalnay2003}.

Due to its ubiquity in scientific applications, data assimilation has been the subject of a very large body of work.  Classically, these techniques are based on linear quadratic estimation, also known as the Kalman Filter.  The Kalman Filter has the drawback of assuming that the underlying system and any corresponding observation models are linear.  It also assumes that measurement noise is Gaussian distributed.  This has been mitigated by practitioners via modifications, such as the Ensemble Kalman Filter, Extended Kalman Filter and the Unscented Kalman Filter and consequently, there has been a recent surge of interest in developing a rigorous mathematical framework for these approaches; see, for instance, \cite{asch,  HMbook2012,  Kalnay2003, KLS, LSZbook2015, ReichCotterbook2015} and the references therein. These works provide a Bayesian and variational framework for the problem, with emphasis on analyzing variational and Kalman filter based methods. It should be noted however that the problems of stability, accuracy and {\it catastrophic filter divergence}, particularly for {\it infinite dimensional chaotic dynamical systems governed by PDE's,} continue to pose serious challenges to rigorous analysis, and are far from being resolved \cite{HMbook2012, BLSZ2013, BLLMCSS2013, tmk2016-1, tmk2016-2}.

An alternative approach to data assimilation, henceforth referred to as the \emph{AOT algorithm,}
has recently been  proposed in \cite{AzouaniTiti2014, AzouaniOlsonTiti2014}, 
which employs a feedback control paradigm
via a \emph{Newtonian relaxation scheme} (\emph{nudging}). 
This was in turn predicated on the notion of finite determining functionals (modes, nodes, volume elements)
for dissipative systems, the rigorous  existence of which was first
established in  \cite{FoiasProdi1967, FoiasTemam1984, FoTi91, JT92}. Assuming that the observations are generated from a continuous dynamical system given by
\[
\dt u = F(u), u(0)=u_0,
\]
the AOT algorithm entails solving an associated system
\be \label{aotsystem}
\dt w =F(w)-\mu (I_hw-I_hu), w(0)=w_0\ (\mbox{arbitrary}),
\ee
where $I_h$ is a finite rank linear operator acting on the phase space, called \emph{interpolant operator}, constructed \emph{solely from observations} on $u$ (e.g. low (Fourier) modes of $u$ or values of $u$ measured in a coarse spatial grid). Here $h$ refers to the size of the spatial grid or, in case of the \emph{modal interpolant}, the reciprocal of $h$ stands for the number of observed modes. Moreover,  $\mu>0$ is the \emph{relaxation/nudging parameter} an appropriate choice of which needs to be made for the algorithm to work. It can then be established that the AOT system
\eqref{aotsystem} is well-posed and its solution \emph{tracks} the solution of the original system asymptotically, i.e. $\|w-u\| \lra 0$ as $t \ra \infty$ in a suitable norm.

\comments{
This  was motivated by earlier work, mainly in the context of  finite-dimensional dynamical systems governed by ordinary differential equations  \cite{Anthes74, hoke, Nijmeijer2001,Thau1973}.  
Notable features of this  approach in \cite{AzouaniOlsonTiti2014} 
includes (i) intuitive conceptual framework based on the idea of finite determining functionals/degrees of freedom and inertial form/manifold applicable to a wide class of dissipative systems \cite{FoiasProdi1967, FoiasTemam1984}
(ii) rigorous analysis guaranteeing global stability (i.e. for $t \in (0,\infty)$) and accuracy (i.e. no catastrophic divergence) (iii) applicability to a wide range of observables (e.g. nodal, volume elements and finite elements), and not just to Fourier modes (iv) availability of explicit conditions on the relaxation (nudging) parameter and the spatial resolution of the observations in order to guarantee convergence of the algorithm to the reference solution of the model equation(s). 
}

Although initially introduced in the context of the two-dimensional Navier-Stokes equations, this was later generalized to include various other models and convergence in stronger norms (e.g. the analytic Gevrey class)  \cite{ANT,BiswasMartinez2017,
FGMW, FarhatJollyTiti2015, FLT, FLT1, FLT2, MarkowichTitiTrabelsi2016, pei}, as well as to more general situations such as discrete in time and  
error-contaminated measurements and to statistical solutions  \cite{BessaihOlsonTiti2015, BFMT, FoiasMondainiTiti2016}. This method has  been shown to perform remarkably well in numerical simulations \cite{AltafTitiKnioZhaoMcCabeHoteit2015, FJJT, GeshoOlsonTiti2015, HOT, HJ2018, larios2019} and has recently been successfully implemented for the first time for efficient dynamical downscaling of a global atmospheric reanalysis \cite{desam2019}. Recent applications include its implementation in \emph{reduced order modeling (ROM)} of turbulent flows to mitigate inaccuracies in ROM \cite{Rebholz2019}, and in
inferring flow parameters and turbulence configurations \cite{DiLeoni1, CHL}.

In this paper, we consider the well-posedness, stability and convergence/tracking property of solutions of the AOT system for the three dimensional Navier-Stokes equations (3D NSE).
Although numerical simulation demonstrating the efficacy of the AOT algorithm for the 3D NSE has recently been demonstrated in \cite{DiLeoni2}, to the best of our knowledge, this is the first such rigorous analytical result for the 3D NSE. In all the cases mentioned before where rigorous analysis is available, including the Navier-Stokes-$\alpha$ models \cite{AB, ANT, FLT2}, one crucially uses the fact that these models are well-posed and regular, i.e. uniform-in-time bound in a higher Sobolev norm (e.g. the $H^1$-norm) is available. These bounds are  used in providing an upper bound on the spatial resolution $h$ of the observations (or lower bound on the number of observed low modes)
necessary for the algorithm to be well-posed, stable and convergent. Additionally, the value of the nudging parameter $\mu$ guaranteeing convergence/tracking property 
also explicitly depends on this uniform bound. 

However, our results here are in the context of (a special class of) Leray-Hopf weak solutions of the 3D NSE for which neither regularity nor uniqueness is known. In fact, due to the recent work of 
\cite{BV} weak solutions of the 3D NSE (but not Leray-Hopf weak solutions), are non-unique. Regardless, we show that the AOT system is well-posed, globally regular (i.e. uniformly bounded in the $H^1$-norm) and tracks the original solution $u$ from which data is collected, provided the observed data satisfy a certain condition.  We emphasize that this condition is imposed on the observed data and no assumption is made on the regularity (or for that matter, uniqueness) of the (Leray-Hopf weak) solution $u$. Our condition is automatically satisfied for a regular solution; however it is unclear whether regularity, or even uniqueness is implied by our condition.

We will give a brief description of our result in the case of the \emph{modal interpolant/observable}, even though our results are applicable to the more general case of a \emph{type 1 interpolant operator} (see Section \ref{sec:prelim} or \cite{AzouaniTiti2014, AzouaniOlsonTiti2014} for definition). Let $A$ denote the Stokes operator (see Section \ref{sec:prelim} or \cite{cf, Temam}) with either the space periodic or homogeneous Dirichlet boundary condition. It is well-known that $A$ is a positive self-adjoint operator with a compact inverse (on appropriate functional spaces as described in Section \ref{sec:prelim}) with eigenvalues
$0 < \lambda_1 < \lambda_2 < \cdots $, with $\lambda_K \ra \infty$. Let the modal interpolant  $P_K$ be orthogonal projection on the (finite dimensional) space spanned by eigenvectors corresponding to eigenvalues $\lambda_1, \cdots , \lambda_K$. Then our condition for the AOT algorithm to be well-posed, admitting a regular solution and possessing the tracking property roughly speaking reads 
(see Theorem \ref{thm:modalregularity} for a more precise formulation)
\be  \label{modalcond}
\exists \ K \in \N, t_0>0\ \mbox{such that}\ 
\lambda_K^{1/4} \gtrsim \frac{1}{\nu}\sup_{[t_0, \infty)}\|P_Ku\|_{H^1}.
\ee
Note first that \eqref{modalcond} depends only on the observed part of the data $P_Ku$ (i.e. the low modes) and does not involve the high modes. Moreover, due to the fact that any Leray-Hopf weak solution of the 3D NSE is unformly bounded in the $L^2$-norm (see \eqref{unifhbound}, \cite{cf}) and the fact that \cite{cf, Temam}
\[
\|P_Ku\|_{H^1} \sim \|A^{1/2}P_Ku\|_{L^2}\le 
\lambda_K^{1/2}\|u\|_{L^2},
\]
 we have
\[
\sup_{t \in [0,\infty)}\|P_Ku\|_{H^1} \lesssim \lambda_K^{1/2}\sup_{t \in [0,\infty)}
\|u\|_{L^2}.
\]
This bound however is much less stringent than the condition
in \eqref{modalcond}. 
A similar condition can be formulated for a more general type 1 interpolant (e.g. volume element) as well (see Theorem \ref{generalh1bound} and Theorem \ref{generalconvergence}). 
Denote 
\be  \label{kinf}
 K_{inf}= \min\{K: K\ \mbox{satisfies}\ \eqref{modalcond}\}.
 \ee
  Suppose that $u$ is globally regular and 
\be  \label{vunifbd}
\sup_{[t_0,\infty)}\|u\|_{H^1} < \infty . 
\ee
Since 
\[
\|P_Ku\|_{H^1} \sim \|A^{1/2}P_Ku\|_{L^2} \le \|u\|_{H^1},
\]
clearly,  in this case, we have $K_{inf} < \infty $.
However, $K_{inf}$ (equivalently $\lambda_{K_{inf}}$)
as defined in \eqref{kinf} may be much smaller than the upper bound provided by \eqref{vunifbd}. It is indeed possible that $u$ is globally regular (i.e. 
$u \in L^\infty ([0,T]; H^1)$ for all $T>0$, yet \eqref{vunifbd} does not hold (although it is known \cite{cf} that if  solution of the 3D NSE is globally regular \emph{for all initial data}, then in fact \eqref{vunifbd} holds). Furthermore, it is unclear whether $K_{inf} < \infty$ implies global regularity or even uniqueness of the Leray-Hopf weak solution $u$. Additionally, it should be noted that
the wave number $K_{inf}$ identified in \eqref{kinf} may be an interesting quantity even in the 2D setting, as it might provide a lower estimate for the determining wave number, at least numerically,  than the ones given for instance in \cite{JT93}, via bounding the $H^1$ norm of the solution as in \eqref{vunifbd}.

The organization of the paper is as follows. In Section \ref{sec:prelim}, we establish the requisite notation and state preliminary results and facts. In Section \ref{sec:mainresults}, we state and prove our main results while in Section \ref{sec:adaptive}, we provide an adaptive version of our akgorithm  which might be useful for computational purposes when the flow is turbulent in certain intervals of time.

\comments{
In this paper we consider a continuous in time feedback-control algorithm 
for data assimilation \cite{AzouaniOlsonTiti2014, FoiasMondainiTiti2016} for the three dimensional Navier-Stokes equations.   We will show that our algorithm is ``accurate" and ``globally stable". 
}

\section{Notation and Preliminaries} \label{sec:prelim}

The 3D incompressible Navier-Stokes equations (3D NSE) on a domain $\Om$ with time independent forcing (assumed for simplicity) \comments{and with either no-slip Dirichlet or periodic boundary conditions,}  is given by
\begin{equation}\label{nse1}
\frac{\partial u}{\partial t}+\nu Au + B(u,u)=f,\ \nabla\cdot u=0 \ for \ t\in(0,\infty)
\end{equation}
with an initial condition $u(0)=u_0$. Concerning the boundary conditions, we either assume that $\Om \subset \R^3$ is  bounded with boundary $\p \Om $  of class $C^2$ and $u\at{\p \Om}=0$ or that $\Om=[0,L]^3$ and $u$ is space periodic with space period $L$ in all variables with space average zero, i.e., 
${\D \int_\Om u=0}$. For simplicity, we also take the body force $f$ to be time-independent.

We briefly introduce the functional framework for \eqref{nse1};
for a more detailed discussion see \cite{cf, Temam}. For $\alpha \ge 0$, we will denote by $H^\alpha(\Om)$ the usual Sobolev space 
of order $\alpha$.
In the case of the \textbf{Periodic Boundary Conditions} we consider $\mathscr{V}$ as the set of all L-periodic trigonometric polynomials from $\mathbb{R}^3$ to $\mathbb{R}^3$ that are divergence free and have zero average. In the case of the \textbf{No-slip Dirichlet Boundary Conditions} we consider $\mathscr{V}$ as the set of all $C^{\infty}$ vector fields from $\Omega$ to $\mathbb{R}^3$ that are divergence free and compactly supported.  Then $H$ is the closure of $\mathscr{V}$ with respect to the norm in $L^2(\Omega)$ and $V$ is the closure of $\mathscr{V}$ with respect to the norm in $H^1(\Omega)$. The inner products in H and V are given by:
\[
(u,v)_{L^2}=\int\limits_{\Omega} u(x)\cdot v(x)dx\ \forall\ u,v\in H
\]
\[
((u,v))=\int\limits_{\Omega}\sum\limits_{i=1}^2 \frac{\partial u}{\partial x_i}\cdot\frac{\partial v}{\partial x_i} dx\ \forall u,v\in V,
\]
and the corresponding norms are given by $|u|=(u,u)^{1/2}$, and $\lVert u \lVert=((u,u))^{1/2}$ respectively. We also denote by 
$\P$ the Leray-Hopf orthogonal projection operator from $L^2(\Om)$ to $H$. The Stokes operator is given by
\[
A v = \P (-\Delta)v, v \in H^2(\Om) \cap V.
\]
We recall that $A$ is a positive self adjoint operator with a compact inverse and $D(A)=\{u\in V:Au\in H\}$. Moreover, there exists a complete orthonormal set of eigenfunctions $\phi_i\in H$ such that $A\phi_i=\lambda_i \phi_i$ where $0<\lambda_1\le\lambda_2 \le \lambda_3 \le \cdots $ are the eigenvalues of $A$ repeated according to multiplicity. In case $u \in D(A^{\alpha}), \alpha \ge 0$ then $u \in H^{2\alpha}(\Om)$ and $\|u\|_{H^{2\alpha}} 
\sim |A^{\alpha}u|$.

 We denote by $H_N$ the space spanned by the first $N$ eigenvectors of $A$ and the orthogonal projection from $H$ onto $H_N$ is denoted by $P_N$. We also recall the 
Poincar\'e inequality
\be \label{poincare}
\lambda_1^{1/2}|v| \le  \|v\|, v \in V.
\ee

Let $V'$ denote the dual space of $V$.
The bilinear continuous operator $B$ from $V \times V$ to $V'$ is defined by
\[
\<B(u,v), w\>=\sum_{i,j} \int_\Om u_i (\p_i v_j) w_j.
\]
The bilinear term B satisfies the orthogonality property 
\[
B(u,w,w)=0\ \forall u,w\in V.
\]
Moreover, the bilinearity of $B$ implies
\be  \label{bilindif}
B(u,u)-B(w,w)=B(u-w,u)-B(w,w-u)=B(\widetilde{w},u)+B(w,\widetilde{w}).
\ee
We recall some well-known bounds on the bilinear term in the 3D case which appear in \cite{robinson, Temam}.
\begin{prop}
If $u,v\in V,\ w\in H$ then
\begin{equation}\label{bilinear3}
|(B(u,v),w)|\le c\|u\|_{L^6}\|\nabla v\|_{L^3}\|w\|_{L^2} \le c \| u\| \|v\|^{1/2} |Av|^{1/2} |w|.
\end{equation}
Moreover, if $u,v,w\in V$, then
\begin{equation}\label{bilinear4}
|(B(u,v),w)|\le c\|u\|_{L^4}\|\nabla v\|_{L^2} \|w\|_{L^4} \le c |u|^{1/4}\|u\|^{3/4} \|v\| |w|^{1/4} \|w\|^{3/4}.
\end{equation}
\end{prop}
\begin{dfn}  \label{def:wksol}
$u$ is said to be a {\em weak solution} of \eqref{nse1}  if for all $T>0$, $u$ belongs to  $L^\infty([0,T]; H) \cap L^2([0,T]; V) \cap C([0,T]; V')$ and satisfies
\[
\dt (u,v) + \nu ((u,v)) +(B(u,u),v)=(f,v)\ \forall v \in V
\ \mbox{and}\ u(0)=u_0.
\]
Additionally, a weak solution is said to be {\em strong solution} if it also belongs to $L^\infty((0,T); V) \cap L^2((0,T); D(A))$.
\end{dfn}
Note that the equality $u(0) =u_0$ makes sense as $u \in C([0,T]; V')$. The Galerkin approximation corresponding to \eqref{nse1} is given by the solution $u_N$ of the following Galerkin system:
\be \label{nse1u}
\begin{split}
\frac{d u_N}{d t}+\nu Au_N + P_N B(u_N,u_N)=P_Nf,\\
 \nabla\cdot u_N=0,\\
u_N(0)= P_Nu(0).
 \end{split}
 \ee
The following theorem due to Leray \cite{cf, robinson, Temam} gives us the existence of weak solutions to (\ref{nse1}) in 3D.
\begin{thm}   \label{thm:Leray}
Let $f\in L_{loc}^2(0,T;V^*)$. Then if $u_0\in H$, there is a weak solution of (\ref{nse1}) such that for any $T>0$,
\[
u\in L^{\infty}(0,T;H)\cap L^2(0,T;V),
\]
and the equation holds as an equality in $L^{4/3}(0,T;V')$.
Moreover, there exists a subsequence $\{u_{N_k}\}$ which converges to a weak solution $u$ weakly in $L^2([0,T]; V)$, strongly in 
$L^2([0,T]; H)$ and in $C([0,T]; V')$.
\end{thm}
\begin{dfn}
Following \cite{LR}, 
we will say that $u$ is a {\em restricted Leray-Hopf weak solution} if it is obtained as a subsequential limit of a Galerkin system where the convergence is as given in Theorem \ref{thm:Leray}. We will denote by  $\frak{W}_{u_0}$ the set of all restricted weak solutions of \eqref{nse1} with initial data $u_0 \in H$ and we denote $\frak{W}=\bigcup_{u_0 \in H} \frak{W}_{u_0}$.
\end{dfn}
It can be shown \cite{cf} that any $u \in \frak{W}$ is in fact a {\em Leray-Hopf} weak solution, i.e., for $a.e.\ t_0 \in [0,\infty)$, including $t_0=0$, 
they satisfy the energy inequality
\beas
\frac12 |u(t)|^2 + \nu \int_{t_0}^t \|u(s)\|^2\, ds &\le &
\frac12 |u(t_0)|^2 + \int_{t_0}^t (f,u(s))\,ds,\quad  t \ge t_0.
\eeas
Additionally, they also satisfy the energy bound
\[
|u(t)|^2 \le e^{-\nu \lambda_1t}|u_0|^2 + \frac{|f|^2}{\nu^2 \lambda_1^2}\lf (1- e^{-\nu \lambda_1t}\rg ), t \ge 0.
\]
Consequently, there exists a time $t_* = t_*(u_0)$ such that
\be  \label{unifhbound}
 |u(t)|^2 \le 2G^2\nu^2\lambda_1\ \forall\ t \ge t_*,
 \ \mbox{where the Grashoff number}\  G := \frac{|f|}{\nu^2 \lambda_1^{3/2}}.
\ee
\begin{rem}
The existence of weak solutions via the Galerkin construction shows that the class $\frak{W}_{u_0}$ is non-empty. However it is presently unknown whether restricted weak solutions (and therefore Leray-Hopf weak solutions) are unique, i.e. whether or not the cardinality of $\frak{W}_{u_0}$ is one. 
\end{rem}

\subsection{Interpolant Operators}

\begin{dfn}
A {\em finite rank, bounded linear operator} $I_h:H\to L^2(\Om)$ is said to be a type-I interpolant observable if there exists a dimensionless constant $c>0$ such that
\begin{equation}\label{type1}
|I_hv| \le c|v|\ \forall\ v \in H\ \mbox{and}\ 
|I_h v-v|\le ch\lVert v \lVert\ \forall\ v \in V.
\end{equation}
\end{dfn}
The orthogonal projection operator $P_K$, also known as the {\em modal interpolant},  provides 
such an example. Indeed, it is easy to check that it  satisfies (\ref{type1}):
\begin{equation}\label{type1k}
|P_Kv| \le |v|\ \forall\ v \in H\ \mbox{and}\ |P_K v-v|\le \frac{1}{\lambda_K^{1/2}}\lVert v \lVert \ \forall\ v \in V.
\end{equation}
Thus \eqref{type1} is satisfied with $h = \frac{1}{\sqrt{\lambda_K}}$. 
Another physically relevant example of a type I interpolant, which is important from the point of view of applications
is the volume elements interpolant  \cite{AzouaniOlsonTiti2014}
given by
\be  \label{volumeint}
I_h(\phi)(x)=\sum\limits_{j=1}^N \bar{\phi_j}\left(\chi_{Q_j}(x)-\frac{h^3}{L^3}\right)\ \ where \ \ \bar{\phi_j}=\frac{1}{h^3}\int_{Q_j} \phi(x) dx,
\ee
where the domain has been divided into equal cubes $Q_j$ of side length $h$. As shown in \cite{FoTi91, JT92, JT93}, the volume element interpolant satisfies \eqref{type1}.


\comments{Since $I_h$ stands for the observation operator, we will denote its range 
by $H_\co \subset L^2(\Omega)$. In applications, $H_\co$ is a finite dimesional subspace.}

\section{Wellposedness and the Tracking Property
of the data assimilation system}  \label{sec:mainresults}
 For the remainder of the paper, {\em  we will assume that $u \in \frak{W}$ is a restricted global weak solution} of \eqref{nse1} corresponding to initial data
 $u_0 \in V$, i.e., $u$  can be approximated by a sequence of solutions $\{u_N\}$ of the Galerkin system in the following way: $u_N\to u$ weakly in $L^2(0,T;V)$, strongly in $L^2(0,T;H)$, and  in $C(0,T;V')$ (equipped with the sup-norm on $[0,T]$).
We begin by describing the AOT data assimilation system  that we consider here.
The observations are given by:
\be   \label{observations}
\mbox{Observation}\ \co= \{I_hu(t)\}_{t \ge 0}
\ee
where $I_h$ is a Type-I interpolant (e.g. either a modal or volume interpolant). Since $I_h$ is of finite rank and $u \in C([0,T]; V')$, the mapping $t \ra I_hu(t)$ from $[0,\infty)$ to $L^2(\Om)$ is continuous. 
\comments{and we assume
\[
\sup_{t\in[0,\infty)} \| I_h u(t) \| < \infty
\]
}
  Our data assimilation algorithm is given by the solution 
  $w$ of the equation
\be \label{nse1a}
\begin{split}
\frac{d w}{d t}+\nu Aw + B(w,w)=f-\mu\mathbb{P}_{\sigma}I_h(w-u)\\
 \nabla\cdot w=0,\\
 w(0)=0.
 \end{split}
 \ee
 We made the choice $w(0)=0$ for specificity. However, 
 the AOT data assimilation system can be initialized by any initial value. Recall that for notational simplicity, we have assumed that $f$ is time independent.
 The Galerkin approximation of w is given by the solution of the equation
\be \label{nse1aa}
\begin{split}
\frac{d w_N}{d t}+\nu Aw_N + P_N B(w_N,w_N)=P_Nf-\mu P_N I_h(w_N-u)\\
 \nabla\cdot w_N=0,\\
w_N(0)= P_Nw(0)=0.
 \end{split}
 \ee
 \subsection{Global Existence of a Weak Solution.}
 We will now show existence of a global (in time) weak solution of 
 \eqref{nse1a}, where the definition of a weak solution is similar to Definition \ref{def:wksol}. As in the case of the 3D NSE, we proceed by establishing {\em a priori} bounds on the Galerkin system \eqref{nse1aa}. Henceforth, by translating time if necessary, we will assume that the weak solution $u \in \frak{W}$ satisfies \eqref{unifhbound} for all $t \ge 0$.
 \begin{thm}  \label{thm:weakexistence}
 Let $u \in \frak{W}$ satisfy \eqref{unifhbound} for all $t \ge 0$ and $I_h$ be any type 1 interpolant satisfying \eqref{type1}. Then, provided
 \be \label{mucond1}
\nu\lambda_1 \le \mu  \le \frac{\nu}{4ch^2}\quad (c \ \mbox{as in}\ \eqref{type1}),
 \ee
there is a weak solution $w$ of (\ref{nse1a}) such that for any $T>0$,
\[
w \in L^{\infty}(0,T;H)\cap L^2(0,T;V)\ \mbox{with}\ 
|w(t)|^2 \lesssim G^2\nu^2 \lambda_1\ \forall t \ge 0,
\]
and the equation holds as an equality in $L^{4/3}(0,T;V')$.
Moreover, there exists a subsequence $\{w_{N_k}\}$ which converges to a weak solution $w$ weakly in $L^2([0,T]; V)$, strongly in 
$L^2([0,T]; H)$ and in $C([0,T]; V')$. 

Additionally, any two strong solutions $w_1, w_2$ 
on the interval $[0,T]$ \\
(i.e. $\sup_{t \in [0,T]}\|w_i(t)\| < \infty , i=1,2$) of \eqref{nse1a} coincide.
 \end{thm}
 \begin{proof}
 We will start by establishing {\em a priori} estimates on the Galerkin system. Taking inner product of \eqref{nse1aa} with $w_N$, we readily obtain (after some elementary algebra)
 \bea
 \lefteqn{\frac12 \dt |w_N|^2 + \nu \|w_N\|^2} \nn \\
 & & = (f,w_N) -\mu|w_N|^2 +\mu(I_hw_N-w_N,w_N)+\mu (I_hu,w_N) \nn \\
 & & \le \frac{|f|^2}{\nu \lambda_1}+\frac{\nu\lambda_1}{4}|w_N|^2
 + \frac{\mu}{4}|w_N|^2 + \mu c h^2 \|w_N\|^2 \nn \\
 & & \qquad \qquad \qquad  + \mu|I_hu|^2 + \frac{\mu}{4}|w_N|^2 - \mu |w_N|^2,  \label{energyineq1}
 \eea
 where to obtain \eqref{energyineq1}, we used Cauchy-Schwartz and Young inequalities, in conjunction with the second inequality in \eqref{type1}. Using the inequalities \eqref{poincare},   \eqref{unifhbound} and the first inequality in \eqref{type1}, we readily obtain
 \be  \label{energyineq}
  \dt |w_N|^2 + \mu|w_N|^2 + \nu \|w_N\|^2 
 \le \frac{|f|^2}{\nu \lambda_1}+c\mu G^2\nu^2 \lambda_1.
 \ee
 Dropping the last term from the left and applying Gronwall inequality together with \eqref{mucond1} and recalling $w_N(0)=0$, we get
 \[
 |w_N|^2 \le \frac{|f|^2}{\mu \nu \lambda_1}+ cG^2 \nu^2\lambda_1
 \lesssim G^2\nu^2\lambda_1.
 \]
 Integrating both sides of \eqref{energyineq} and insering the above bound , we immediately obtain
 \[
 \nu \int_0^T \|w_N\|^2 \lesssim \lf(1+\mu T\rg)G^2\nu^2\lambda_1.
 \]
 The remainder of the proof is similar to the proof of existence of  weak solutions of the 3D NSE \cite{cf, Temam}.
 
 We will now prove uniqueness of strong solutions. Observe that
 $\|B(w,w)\|_{V'} \lesssim \|w\|_{L^4}^2 \le |w|^{1/2}\|w\|^{3/2}$.  It is now easy to see from \eqref{nse1a} that if $w$ is a strong solution on $[0,T]$, then $\dt w \in L^2([0,T];V')$ and consequently, $|w(t)|^2$ is differentiable $a.e.$ on $[0,T]$. 
 Let $\tw=w_2-w_1$. Then $\tw$ satisfies
 \be  \label{diffeqn}
 \dt \tw + \nu A\tw + B(w_2,\tw) + B(\tw,w_1)=-\mu I_h\tw.
 \ee
 Let 
 \be  \label{maxdef}
 M= \sup_{[0,T]} \|w_1\|< \infty. 
 \ee
 By taking inner product with $\tw$ with \eqref{diffeqn}, using the properties of type 1 interpolant in \eqref{type1} and the estimate of the nonlinear term in \eqref{bilinear4}, we readily obtain
 \beas
 \lefteqn{\dt |\tw|^2 + \nu \|\tw\|^2 } \\
 & & \le c |\tw|^{1/2}\|\tw\|^{3/2}\|w_1\| -\mu|\tw|^2 + \mu ch^2\|\tw\|^2 + \frac{\mu}{2}|\tw|^2 \\
 & & \le \frac{cM}{\nu^3}|\tw|^2 + \frac{\nu}{2}\|\tw\|^2 
 + \frac{\mu}{2}\|\tw\|^2 - \frac{\mu}{2}|\tw|^2,
 \eeas
 where to obtain the inequality in the line above, we used Young's inequality together with \eqref{maxdef} and the condition on $\mu$ in \eqref{mucond1}. Since $\tw(0)=0$, using Gronwall inequality, we readily conclude that $\tw=0$ on $[0,T]$, i.e., $w_1$ and $w_2$ coincide on $[0,T]$.
 \end{proof}
 \begin{rem}
 Observe that in the proof of the uniqueness presented above, the bound in $V$ of only one of the two solutions appears explicitly. Indeed, one can prove a weak-strong uniqueness as in the case of the 3D NSE, a result due to J. Sather and J. Serrin
 \cite{SS}. More precisely, if there exists a strong solution $w$ of \eqref{nse1a} on $[0,T]$, then it coincides with any other Leray-Hopf weak solution of \eqref{nse1a}. The proof is similar to that of the Sather-Serrin result \cite{SS}.
 
 \end{rem}
 \subsection{Global Existence of a Strong Solution
 and tracking property.}
Thus far, no assumption on the solution $u$  was necessary to establish existence of weak solution of \eqref{nse1a}. In order to ensure global existence of a (hence the) regular solution of \eqref{nse1a} and to establish the tracking property (i.e. to show that it tracks $u$ asymptotically), we need  to impose condition(s) on the {\em observed data} coming from the solution $u$. For clarity of exposition, we will consider the case of modal interpolant first before proceeding to a more general type 1 interpolant.
\subsubsection{Modal Interpolant Case (i.e. $I_h=P_K$)}
\hspace{\fill}\\

Before we proceed, we first note that $\sup_{t \ge 0}\|P_Ku\| < \infty $. Indeed, 
\[
\|P_Ku\| \le \lambda_K^{1/2}|P_Ku| 
\lesssim \lambda_K^{1/2}G^2\nu^2\lambda_1,
\]
where the last inequality follows from \eqref{unifhbound}.
However, as we will see below (see Remark \ref{rem:significancemodal} and \eqref{datamodalcond}), this bound is insufficient to guarantee that $w$ tracks $u$. We require a more stringent bound as given in \eqref{mumodalcond1} or equivalently, as in \eqref{datamodalcond}.
\begin{thm}  \label{thm:modalregularity}
Suppose $I_h=P_K$ be modal interpolant which satisfies (\ref{type1k}).
Let $0 < T \le \infty$ and denote
\be  \label{mkdef}
M_K^2=M_{K,u}^2:= 8 \left(\frac{|f|^2}{\nu^2\lambda_1}  +  \sup_{t\in [0,T)}\|P_K(u)\|^2\right).
\ee
Assume that
\be  \label{mumodalcond1}
\nu \max\left\{\frac{2cM_K^4}{\nu^4},\lambda_1\right\}\le \mu \le 
 \frac{\nu \lambda_K}{4}.
 \ee
 Then any  weak solution of \eqref{nse1a} constructed in Theorem \ref{thm:weakexistence} as a subsequential limit of $w_N$ satisfying \eqref{nse1aa} is regular on $[0,T]$, i.e., it satisfies
\be   \label{wh1bound}
\|w(t)\| \le M_K, t \in [0,T].
\ee
\comments{
Consequently, for a choice of $\mu$ is an \eqref{mumodalcond1}, any weak solution $w_1$ of \eqref{nse1a} coincides with the regular solution $w$  on the interval $[0,T]$.
}
\end{thm}
\comments{
\begin{rem}
This implies $\lambda_K\ge 8c M_K^4 \nu^{-3}$.
\end{rem}
}
\begin{proof}
As is customary, we will obtain {\em a priori} estimates on the Galerkin system and then pass to the limit.
We begin by rearranging (\ref{nse1aa}) with the assumption 
 $N\ge K$ to first obtain
\[
\frac{d w_N}{d t}+\nu Aw_N + P_N B(w_N,w_N)=f-\mu P_N(P_K(w_N)-w_N) + \mu P_K(u) - \mu w_N.
\]
Now taking the inner product with $Aw_N$ yields
\begin{align}\label{eq888}
\lefteqn{\frac{1}{2}\frac{d}{dt}\| w_N \| ^2 + \nu|Aw_N|^2 + \mu \|w_N\|^2} \nn \\
 &= (f,Aw_N) -(B(w_N,w_N),Aw_N)  \nn \\
& \qquad \qquad -\mu (P_K(w_N)-w_N,Aw_N) + \mu(P_K(u),Aw_N).
\end{align}
Each term on the right hand side is estimated below. First, using Cauchy-Schwartz and Young's inequality, we have
\[
|(f,Aw_N)|\le \frac{1}{\nu}|f|^2 + \frac{\nu}{4} |Aw_N|^2.
\]
Next, using \eqref{bilinear3} and Young's inequality, we obtain
\[
|(B(w_N,w_N),Aw_N)|\le c\|w_N\|^{3/2}|Aw_N|^{3/2}\le \frac{c}{\nu^3} \|w_N\|^6 + \frac{\nu}{4} |Aw_N|^2.
\]
Observe now that from \eqref{type1k} and Young's inequality, we obtain
\begin{align*}
&\mu|(P_K(w_N)-w_N,Aw_N)|\le \mu|P_K(w_N)-w_N||Aw_N|\le \mu\lambda_K^{-1/2} \|w_N\||Aw_N|\\
&\ \ \ \ \ \ \ \ \ \ \ \ \ \ \ \ \ \ \ \ \ \ \ \ \ \ \ \ \ \ \ \le \frac{\mu^2}{\nu\lambda_K}\|w_N\|^2 + \frac{\nu}{4}|Aw_N|^2\le \frac{\mu}{4}\|w_N\|^2 + \frac{\nu}{4}|Aw_N|^2.
\end{align*}
Moreover,
\beas
& & \mu|(P_K(u),Aw_N)|=\mu|(A^{1/2}P_K(u),A^{1/2}w_N)| \\
& & \le \mu \|P_K(u)\|\|w_N\| \le \mu\|P_K(u)\|^2 + \frac{\mu}{4}\|w_N\|^2.
\eeas
Inserting these estimates into (\ref{eq888}), we obtain
\be  \label{dblnormestmodal}
\frac{d}{dt}\|w_N\|^2 + (\mu-\frac{c}{\nu^3} \|w_N\|^4)\|w_N\|^2 \le \frac{2}{\nu}|f|^2 +2\mu\|P_K(u)\|^2.
\ee
Let $[0,T_1]$ be the maximal interval on which $\|w_N(t)\|\le M_K$ holds for $t\in[0,T_1]$ where $M_K$ as in \eqref{mkdef}.
 Note that $T_1>0$ exists because we have $w_N(0)=0$. Assume  that $T_1 < T$. Then by continuity, we must have $\|w_N(T_1)\|=M_K$.
 Using the lower bound for  $\mu$ in 
 \eqref{mumodalcond1}, for all $t \in [0,T_1]$, we obtain
 \[
 \frac{d}{dt}\|w_N\|^2 + \frac{\mu}{2}\|w_N\|^2 \le \frac{2}{\nu}|f|^2 +2\mu \sup_{s \in [0,T]}\|P_K(u(s))\|^2.
\]
 Since $w_N(0)=0$, by Gronwall inequality we immediately obtain
\[
\|w_N\|^2 \le \frac{4}{\nu^2\lambda_1} |f|^2 + 4\sup_{s\in[0,T)} \|P_K(u(s))\|^2\le\frac{1}{2}M_K^2\ \forall t \in [0,T_1].
\]
This contradicts $\|w_N(T_1)\|=M_K$. Therefore we conclude $T_1\ge T$
and consequently, $\|w_N(t)\| \le M_K$ for all $t \in [0,T]$. Passing to the limit as $N \ra \infty$, we obtain the desired conclusion for $w$.
\end{proof}

\begin{thm}  \label{thm:modalconvergence}
\comments{
Suppose $I_h=P_K$ is a modal interpolant and satisfies (\ref{type1k}), $\mu\ge \max\{2cM_K^4\nu^{-3},1\}$, $\lambda_K\ge\max\{2\mu\nu^{-1},4\mu\}$.  }
Assume that the hypotheses of Theorem \ref{thm:modalregularity} hold. Let $\tw=w-u$.
Then $|\tw(t)|^2 \le e^{-\frac{\mu}{2}t}|\tw(0)|^2$ for all 
$t \in [0,T]$. In particular, if in the statement of theorem 
\ref{thm:modalregularity} $T=\infty$, then
\[
\lim_{t\to\infty}|\widetilde{w}(t)|^2=0.
\]
\end{thm}
\comments{
\begin{rem}
This implies $\lambda_K\ge 2c M_K^4 \nu^{-3}\max\{2\nu^{-1},4\}$.
\end{rem}
}
\begin{proof}
Assume $N\ge K$ and $\widetilde{w}_N$ satisfies the following,
\begin{align*}
\frac{d\tw_N}{dt} + vA\tw_N + P_N B(w_N,w_N)-P_NB(u_N,u_N) = -\mu P_K(\tw_N) + \mu P_K(u-u_N)
\end{align*}
which can be rearranged to
\begin{align*}
\frac{d\tw_N}{dt} + vA\tw_N + P_NB(\tw_N,w_N)+P_NB(u_N,\tw_N) & = -\mu Q_K(\tw_N) -\mu \tw_N\\
& + \mu P_K(u-u_N).
\end{align*}
We take the inner product with $\tw_N$
\begin{align}\label{eq8888}
\frac{1}{2}\frac{d}{dt}|\tw_N|^2 + \nu\|\tw_N\|^2 + \mu |\tw_N|^2 &= -(B(\tw_N,w_N),\tw_N) -\mu(Q_K(\tw_N),\tw_N)\nn \\
& + \mu ( P_K(u-u_N),\tw_N)
\end{align}
and estimate each term on the right hand side as follows:
\begin{align*}
&|(B(\tw_N,w_N),\tw_N)|\le c|\tw_N|^{1/2} \|w_N\| \|\tw_N\|^{3/2}\le \frac{c}{\nu^3}\|w_N\|^4|\tw_N|^2 + \frac{\nu}{2}\|\tw_N\|^2\\
& \mu|(Q_K(\tw_N),\tw_N)|\le \mu |Q_K(\tw_N)||\tw_N|\le \frac{\mu}{\lambda_K^{1/2}}\|\tw_N\| |\tw_N|\\
& \ \ \ \ \ \ \ \ \ \ \ \ \ \ \ \ \ \ \ \ \ \ \ \ \ \ \ \ \ \le \frac{\mu}{\lambda_K}\|\tw_N\|^2 + \frac{\mu}{4} |\tw_N|^2 \le \frac{\nu}{2}\|\tw_N\|^2 + \frac{\mu}{4}|\tw_N|^2 \\
&\mu |(P_K(u-u_N),\tw_N)|\le \mu |P_K(u-u_N)||\tw_N|\le \mu |P_K(u-u_N)|^2 + \frac{\mu}{4}|\tw_N|^2.
\end{align*}
Inserting the estimates into (\ref{eq8888}),
\[
\frac{d}{dt}|\tw_N|^2 + (\mu-\frac{c}{\nu^3}\|w_N\|^4) |\tw_N|^2 \le \mu |P_K(u-u_N)|^2.
\]
Since $\mu$ satisfies \eqref{mumodalcond1}, we get
\[
\frac{d}{dt}|\tw_N|^2 + \frac{\mu}{2} |\tw_N|^2 \le \mu |P_K(u-u_N)|^2.
\]
Applying Gronwall, for all $t \in [0,T]$, we get
\[
|\tw_N(t)|^2\le e^{-(\mu/2)t}|\tw(0)|^2 + 2\sup_{t\in[0,T]}|P_K(u-u_N)|^2.
\]
Recall $u_N\to u$ in $C(0,T;V')$ and $\phi_i$ is the ith eigenvector associated with A. Therefore,
\[
\liminf_{N\to\infty}|P_K(u-u_N)|\le \liminf_{N\to\infty}\sum\limits_{i=1}^{K} |(u-u_N,\phi_i)|=0.
\]
Since $\tw_N$ converges to $\tw$ weakly,
\[
|\tw(t)|^2\le e^{-(\mu/2)t}|\tw(0)|^2.
\]
which proves the result.
\end{proof}
\begin{rem}
The fact that $u$ is a restricted Leray-Hopf weak solution, i.e., the fact that $u$ is obtained as a suitable limit of Galerkin (or some other appropriate limiting) procedure, is used only in the proof of Theorem \ref{thm:modalconvergence}. This is due to the fact that it is unknown whether for an arbitrary Leray-Hopf weak solution, the quantity $|u(t)|^2$ (and hence $|\tw(t)|^2=|(w-u)|^2$ is differentiable, although $|w(t)|^2$ is due to its being a strong solution. It should be noted that Theorem \ref{thm:modalregularity} and Theorem \ref{thm:modalconvergence} applies just as well to Leray-Hopf weak solutions obtained via other limiting procedures, such as limits in appropriate sense of solutions to the 3D Leray-$\alpha$ NSE as $\alpha \ra 0$. 
\end{rem}

\begin{rem}  \label{rem:significancemodal}
 Note that the conclusions of Theorem \ref{thm:modalregularity} and Theorem \ref{thm:modalconvergence} rest on the choice of $\mu$ satisfying \eqref{mumodalcond1}. This is possible provided there exists $K$ such that
\be \label{datamodalcond}
\lambda_K \gtrsim \max\left\{\frac{M_K^4}{\nu^4}, \lambda_1\right\},
\ee
where $M_K$ is as defined in \eqref{mkdef}. We emphasize {\em that this condition is expressed purely in terms of the observed data which in this case are the low modes of the solution} $u$ and does not involve information on the unknown high modes.

Suppose now that the solution $u$ is regular on $[0,T]$. For initial data $u_0 \in V$, it is well known that this happens if for some $\frac12 < \theta \le 1$ (here we take $\theta \le 1$ as $u_0 \in V=D(A^{1/2})$)
\[  
\sup_{t \in [0,T]} |A^{\theta/2}u|=M_\theta < \infty.
\]
In this case, 
\[
\|P_Ku\|=|P_KA^{1/2}u|=\lambda_K^{(1-\theta)/2}|A^{\theta/2}u|
\le \lambda_K^{(1-\theta)/2}M_\theta.
\]
Therefore, 
\[
M_K^2 \lesssim \left(\frac{|f|^2}{\nu^2\lambda_1}+\lambda_K^{(1-\theta)}M_\theta^2\right).
\]
Since $\theta > \frac{1}{2}$, a choice of $\lambda_K$ satisfying \eqref{datamodalcond} 
is indeed possible if $K$ if $K$ is chosen large enough. In the borderline case $\theta = \frac12$, by proceeding in an analogous manner, we get that \eqref{datamodalcond} can be satisfied if 
$M_\theta =\sup_{t \in [0,T]}|A^{1/4}u|$ is small. It should be noted that it is well-known that in case $f$ is small and $|A^{1/4}u_0|$ is small, then the solution $u$ is globally regular and additionally $|A^{1/4}u|$ remains small for all times. Thus we conclude that \eqref{datamodalcond} holds for sufficiently large $K$ when the solution is regular. However, it is unclear whether or not \eqref{datamodalcond} implies regularity, or even uniqueness, of the solution $u$. However, we show below that  this condition implies \emph{asymptotic uniqueness}, i.e., $K$ satisfying \eqref{datamodalcond} is \emph{asymptotically determining.}
\end{rem}
\begin{thm}  \label{thm:detmodes}
Let $u_1, u_2$ be two restricted Leray-Hopf weak solutions with $M_{K,u_i}, i=1,2$ defined as in \eqref{mkdef}.
Assume moreover that on $[0,\infty)$, we have
\be \label{detmodecond}
\lambda_K \gtrsim \max\left\{\frac{M_{K,u_1}^4}{\nu^4},\frac{M_{K,u_2}^4}{\nu^4},
 \lambda_1\right\},
\ee
where $M_{K,u_i}, i=1,2$ as defined in \eqref{mkdef} with 
$T=\infty$. If 
\be  \label{detmodecond1}
\lim_{t \ra \infty} |P_K(u_1-u_2)| =0,
\ee
then $\lim_{t \ra\infty} |u_1(t)-u_2(t)|=0$.
\end{thm}
\begin{proof}
Let $w_1$ and $w_2$ be two strong solutions of the data assimilation equation \eqref{nse1a} corresponding to $u_1$ and $u_2$ for $\mu$ satisfying \eqref{mumodalcond1} for both $u_1$ and $u_2$.. 
Denote $\tw_i=w_i-u_i, i=1,2$. Then by Theorem \ref{thm:modalconvergence}, $\lim_{t \ra \infty}|\tw_i|=0$. 
Let $\tw=w_1-w_2$. Proceeding exactly as in the proof of Theorem \ref{thm:modalconvergence}, and noting that $\|w_i\| \le M_K$ and that 
$\mu$ satisfies \eqref{mumodalcond1}, we conclude
\[
\dt |\tw|^2 +\frac{\mu}{2}|\tw|^2 \le \mu|P_K(u_2-u_1)|^2,
\]
which yields, upon integrating between $s$ to $T$ that
\[
|\tw(T)|^2 \lesssim e^{- \frac{\mu}{2}(T-s)}|\tw(s)|^2 
+ \sup_{t \in [s,T]} |P_K(u_2-u_1)|.
\]
Letting $T \ra \infty$ and using \eqref{detmodecond1}, we conclude that $\lim_{t \ra \infty}|\tw|=0$. Thus, $\lim_{t \ra \infty}|u_1-u_2|=0$.

\end{proof}
\begin{rem}
A different notion of a time-varying determining wave number for two Leray-Hopf weak solutions, based on the Littlewood-Paley decomposition and applicable only to the space periodic setting, is provided in \cite{chedai}. However, the time varying determining wave number defined there is not asymptotically determining in our sense. To be precise, their result requires the two solutions to belong to the global attractor, and moreover, the low Fourier modes, as defined by this time varying wave number, must be identical for all times for the solutions to coincide. Thus they are not asymptotically determining.
 \end{rem}

\subsubsection{General Type 1 Interpolant}
\hspace{\fill}\\

In this section, we assume that $u$ and $w$ satisfy the space periodic boundary condition. Thus, the Stokes operator 
$A=(-\Delta)$ on $V=H^1(\Om)$. Moreover, 
 we also assume that in addition to \eqref{type1}, $I_h$ also satisfies the condition $Ran(I_h)\subset V=H^1(\Om)$ and
\be  \label{type1a}
 \|I_h v\|\le C \|v\|\ \forall v \in V.
\ee
This is clearly satisfied for the modal interpolant $P_K$. In case of the volume interpolant, one may  apply a suitable mollification procedure to obtain a modified volume interpolant,  $\tilde{I}_h$, that satisfies both \eqref{type1} and \eqref{type1a}. In fact, this is achieved by replacing the term $\left(\chi_{Q_j}(x)-\frac{h^3}{L^3}\right)$
by $\lf(\psi_j(x) - \int_{\Om} \psi_j (y)dy\rg)$, where 
$\psi_j= \rho_\epsilon \ast \chi_{Q_j}$ for
\[  \rho(\xi)= \left\{
\begin{array}{ll}
      K_0 \exp\left(\frac{1}{1-|\xi|^2}\right) & for\ |\xi| < 1 \\
      0 & for\ |\xi|\ge 1 \\
\end{array} 
\right. \]
and 
\[
(K_0)^{-1}=\int\limits_{|\xi|<1} \exp\left(\frac{1}{1-|\xi|^2}\right)d\xi
\]
The mollification parameter $\epsilon$ is chosen to be a fraction of h. $\tilde{I}_h$ is a $C^{\infty}$ function and it can be show that it satisfies (\ref{type1a}). For more details see the appendix of \cite{AzouaniOlsonTiti2014} which proves the corresponding result for the type 2 case.

\begin{thm}\label{generalh1bound}
Assume that $u$ as in \eqref{nse1} and $w$ as in \eqref{nse1a} satisfy the space periodic boundary condition.
Suppose $I_h$ is a general type 1 interpolant and satisfies (\ref{type1}). Let $0<T\le \infty$ and denote
\be\label{Mhdef}
M_h^2=M_{h,u}^2:= 8\left(\frac{1}{\nu^2\lambda_1} |f|^2 +  \sup_{t\in [0,T)}\|I_h(u)\|^2\right).
\ee
Assume that
\be\label{mucondg}
\max\{\frac{2cM_h^4}{\nu^3}, \nu\lambda_1\}\le \mu \le \frac{ \nu}{4ch^2}.
\ee
 Then any  weak solution of \eqref{nse1a} constructed in Theorem \ref{thm:weakexistence} as a subsequential limit of $w_N$ satisfying \eqref{nse1aa} is regular on $[0,T]$, i.e., it satisfies
\be   \label{wh1boundtype1}
\|w(t)\| \le M_h, t \in [0,T].
\ee
\end{thm}

\begin{proof}
We proceed as in Theorem \ref{thm:modalregularity} by rearranging (\ref{nse1aa})
\[
\frac{d w_N}{d t}+\nu Aw_N + P_N B(w_N,w_N)=f-\mu P_N(I_h(w_N)-w_N) + \mu P_N I_h(u) - \mu w_N
\]
and taking the inner product with $Aw_N$
\bea
\lefteqn{\frac{1}{2}\frac{d}{dt}\| w_N \| ^2 + \nu|Aw_N|^2 + 
\mu \|w_N\|^2 }\nn \\
& & = (f,Aw_N) -(B(w_N,w_N),Aw_N) \nn \\
& & \qquad \qquad  -\mu (I_h(w_N)-w_N,Aw_N) + \mu(I_h(u),Aw_N).
\label{eq8888alt}
\eea
Each term on the right hand side is estimated below as in Theorem \ref{thm:modalregularity}. First, we have by Cauchy-Schwartz and Young's inqualities,
\[
|(f,Aw_N)|\le |f||Aw| \le \frac{1}{\nu}|f|^2 + \frac{\nu}{4} 
|Aw_N|^2.
\]
Next, the nonlinear term is estimated using \eqref{bilinear3}
and Young's inequalities as
\[
|(B(w_N,w_N),Aw_N)|\le c\|w_N\|^{3/2}|Aw_N|^{3/2}\le \frac{c}{\nu^3} \|w_N\|^6 + \frac{\nu}{4} |Aw_N|^2.
\]
Next, using \eqref{type1a}, 
\begin{align*}
&\mu|I_h(w_N)-w_N,Aw_N)|\le \mu|I_h(w_N)-w_N||Aw_N| \\
& \le \mu c h \|w_N\||Aw_N|
 \le \frac{(\mu ch)^2}{\nu}\|w_N\|^2 + \frac{\nu}{4}|Aw_N|^2\\
& \le \frac{\mu}{4}\|w_N\|^2 + \frac{\nu}{4}|Aw_N|^2,
\end{align*}
where to obtain the last inequality, we used \eqref{mucondg}.
Observe now that since $A=(-\Delta)$ in the space periodic case, we can integrate by parts to obtain
\[
\mu|(I_h(u),Aw_N)|\le \mu \|I_h(u)\|\|w_N\| \le \mu\|I_h(u)\|^2 + \frac{\mu}{4}\|w_N\|^2.
\]
Inserting the above estimates into (\ref{eq8888alt}) we obtain,
\[
\frac{d}{dt}\|w_N\|^2 + (\mu-\frac{c}{\nu^3} \|w_N\|^4)\|w_N\|^2 \le \frac{2}{\nu}|f|^2 +2\mu\|I_h(u)\|^2.
\]
Let $[0,T_1]$ be the maximal interval on which $\|w_N(t)\|\le M_h$ holds for $t\in[0,T_1]$ where $M_h$ as in (\ref{Mhdef}).
 Note that $T_1$ exists because we have $w_N(0)=0$. Assume that $T_1 < T$. Using the lower bound for $\mu$ in (\ref{mucondg}) and the Gronwall inequality we obtain
\[
\|w_N\|^2 \le \frac{4}{\nu^2\lambda_1} |f|^2 + 4\sup_{s\in[0,T]} \|I_h(u(s))\|^2=\frac{1}{2}M_h^2\ \forall\  t\in[0,T_1].
\]
Arguing as in Theorem \ref{thm:modalregularity} by contradiction, we obtain the desired conclusion for $w$, i.e., $\|w(t)\|\le M_h$ for all $t\in[0,T]$.
\end{proof}

We now can deduce  the following result regarding the tracking property of $w$.
\begin{thm}\label{generalconvergence}
Assume that the hypotheses of Theorem \ref{generalh1bound} hold. Let $\tw =w-u$. Then $|\tw (t)|^2 \le e^{\frac{-\mu}{2} t}|\tw (0)|^2$ for all $t\in[0,T]$. In particular, if in the statement of Theorem \ref{generalh1bound}, we have $T=\infty$, then
\[
\lim_{t\to\infty}|\widetilde{w}(t)|^2=0.
\]
\end{thm}
\begin{proof}
$\widetilde{w}_N$ satisfies the following,
\begin{align*}
\frac{d\tw_N}{dt} + vA\tw_N + P_N B(w_N,w_N)-P_NB(u_N,u_N) = -\mu P_N I_h(\tw_N)  + \mu P_N I_h(u-u_N)
\end{align*}
which can be rearranged to
\begin{align*}
\frac{d\tw_N}{dt} + vA\tw_N + P_NB(\tw_N,w_N)+P_NB(u_N,\tw_N) &= -\mu(I_h(\tw_N)-\tw_N) -\mu \tw_N\\
&+ \mu P_N I_h(u-u_N).
\end{align*}
We take the inner product with $\tw_N$
\begin{align}\label{eq88888}
\frac{1}{2}\frac{d}{dt}|\tw_N|^2 + \nu\|\tw_N\|^2 + \mu |\tw_N|^2 &= -(B(\tw_N,w_N),\tw_N) -\mu(I_h(\tw_N)-\tw_N,\tw_N)\nn \\
&+ \mu (I_h(u-u_N),\tw_N)
\end{align}
and estimate each term on the right hand side as follows:
\begin{align*}
&|(B(\tw_N,w_N),\tw_N)|\le c|\tw_N|^{1/2} \|w_N\| \|\tw_N\|^{3/2}\le \frac{c}{\nu^3}\|w_N\|^4|\tw_N|^2 + \frac{\nu}{2}\|\tw_N\|^2\\
& \mu|(I_h(\tw_N)-\tw_N,\tw_N)|\le \mu |I_h(\tw_N)-\tw_N||\tw_N|\le \mu ch\|\tw_N\| |\tw_N|\\
& \ \ \ \ \ \ \ \ \ \ \ \ \ \ \ \ \ \ \ \ \ \ \ \ \ \ \ \ \ \le \mu c^2h^2\|\tw_N\|^2 + \frac{\mu}{4} |\tw_N|^2 \le \frac{\nu}{2}\|\tw_N\|^2 + \frac{\mu}{4}|\tw_N|^2\\
&\mu|( I_h(u-u_N),\tw_N)|\le \mu|I_h(u-u_N)||\tw_N|\le \mu |I_h(u-u_N)|^2 + \frac{\mu}{4}|\tw_N|^2.
\end{align*}
Inserting the estimates into (\ref{eq88888}),
\[
\frac{d}{dt}|\tw_N|^2 + (\mu-\frac{c}{\nu^3}\|w_N\|^4) |\tw_N|^2 \le \mu |I_h(u-u_N)|^2
\]
and since $\mu$ satisfies (\ref{mucondg}) we get
\[
\frac{d}{dt}|\tw_N|^2 + \frac{\mu}{2} |\tw_N|^2 \le \mu |I_h(u-u_N)|^2.
\]
Applying Gronwall, for all $t\in[0,T]$, we get
\[
|\tw_N(t)|^2\le e^{-(\mu/2)t}|\tw(0)|^2 + 2\sup_{t\in[0,t]}|I_h(u-u_N)|^2.
\]
Recall $u_N\to u$ in $C(0,T;V')$ and the range of $I_h$ is a finite dimensional vector space with a basis $\{\psi\}$. Therefore,
\[
\liminf_{N\to\infty}|I_h(u-u_N)|\le \liminf_{N\to\infty}\sum\limits_{i=1}^{K'} |(u-u_N,\psi_i)|=0.
\]
Since $\tw_N$ converges to $\tw$ weakly,
\[
|\tw(t)|^2\le e^{-(\mu/2)t}|\tw(0)|^2
\]
which proves the result.
\end{proof}

\begin{rem}
A choice of $\mu$ satisfying \eqref{mucondg} exists provided the condition
\be  \label{type1mucond}
\max\{\frac{2cM_h^4}{\nu^4}, \lambda_1\}\lesssim \frac{ 1}{h^2}
\ee
holds. Due to \eqref{type1a}, this is clearly satisfied for sufficiently small $h$ if $u$ is regular and 
$\sup_{t \in [t_0,\infty)} \|u\| < \infty $. Thus global regularity and uniform boundedness in $V$ of $u$ guarantees the existence of a globally regular solution for the AOT algorithm \eqref{nse1a} and the unique solution $w$ tracks $u$.
\end{rem}
We now show that Theorem \ref{generalconvergence} implies  the existence of asymptotically determining volume elements, 
similar to the modal case.

\begin{thm}  \label{thm:detmodesgeneral}
Let $u_1, u_2$ be two restricted Leray-Hopf weak solutions with $M_{h,u_i}, i=1,2$ defined as in \eqref{Mhdef}.
Assume moreover that on $[0,\infty)$, we have
\be \label{detmodecondgeneral}
h^{-2} \gtrsim \max\left\{\frac{M_{h,u_1}^4}{\nu^4},\frac{M_{h,u_2}^4}{\nu^4},
 \lambda_1\right\},
\ee
where $M_{h,u_i}, i=1,2$ as defined in \eqref{Mhdef} with 
$T=\infty$. If 
\be  \label{detmodecond1g}
\lim_{t \ra \infty} |I_h(u_1-u_2)| =0,
\ee
then $\lim_{t \ra\infty} |u_1(t)-u_2(t)|=0$.
\end{thm}
\begin{proof}
Let $w_1$ and $w_2$ be two strong solutions of the data assimilation equation \eqref{nse1a} corresponding to $u_1$ and $u_2$ for $\mu$ satisfying \eqref{mucondg} for both $u_1$ and $u_2$.. 
Denote $\tw_i=w_i-u_i, i=1,2$. Then by Theorem \ref{generalconvergence}, $\lim_{t \ra \infty}|\tw_i|=0$. 
Let $\tw=w_1-w_2$. Proceeding exactly as in the proof of Theorem \ref{generalconvergence}, and noting that $\|w_i\| \le M_h$ and that 
$\mu$ satisfies \eqref{mucondg}, we conclude
\[
\dt |\tw|^2 +\frac{\mu}{2}|\tw|^2 \le \mu|I_h(u_2-u_1)|^2,
\]
which yields, upon integrating between $s$ to $T$ that
\[
|\tw(T)|^2 \lesssim e^{- \frac{\mu}{2}(T-s)}|\tw(s)|^2 
+ \sup_{t \in [s,T]} |I_h(u_2-u_1)|.
\]
Letting $T \ra \infty$ and using \eqref{detmodecond1g}, we conclude that $\lim_{t \ra \infty}|\tw|=0$. Thus, $\lim_{t \ra \infty}|u_1-u_2|=0$.

\end{proof}

\section{Adaptive Algorithm}  \label{sec:adaptive}
Since \eqref{nse1a} becomes stiff for larger values of $\mu$, in this section we define an adaptive algorithm so that the value of the nudging parameter can be adjusted to a higher value only in the time intervals where the flow is turbulent.
 This data assimilation algorithm is iteratively defined by
\be \label{nse1g}
\begin{split}
\frac{d w}{d t}+\nu Aw + B(w,w)=f-\mu_{k+1}P_N(w-u),\ t\in(T_k,T_{k+1}]\\
 \nabla\cdot w=0,\\
 w(T_k)=\lim_{t\to T_k} w(t),\ w(0)=0.
 \end{split}
 \ee
 for $k\in \{0,1,...,j\}$.
\begin{thm}\label{wboundg}
Suppose $I_h=P_K$ is a modal interpolant and satisfies (\ref{type1k}). Denote for $k\in \{0,1,...,j\}$

\be\label{Mkitdef}
M_{k+1}^2=\max\bigg\{\|w(T_k)\|^2, \left(\frac{4|f|^2}{\nu^2\lambda_1} + 2\widetilde{M}_{k+1}^2\right)\bigg\}
\ee
and
\be\label{Mktildedef}
\widetilde{M}_{k+1}=\sup_{t\in[T_k,T_{k+1}]} \|P_K u(t)\|.
\ee
Assume that
\be\label{adaptiveassumption}
\sup_{t\in[T_0,T_{j+1}]} \left(\frac{32|f|^4}{\nu^4\lambda_1^2} + 8\|P_K u(t)\|^4\right) \lambda_K^{-1} \le \frac{\nu^4}{16c}
\ee
and 
\be\label{adaptivemucond}
 \max\big\{\frac{2c}{\nu^3}M_{k}^4,\nu\lambda_1\big\}\le \mu_{k}\le \frac{\nu\lambda_N}{8},\ k\in \{0,1,...,j+1\}.
\ee
 Then the solution of (\ref{nse1g}) satisfies
\[
\|w(t)\|^2 \le \frac{4|f|^2}{\nu^2\lambda_1} + 2\sup_{t\in[T_0,T_{j+1}]} \|P_K u(t)\|^2\le \frac{\nu^2\lambda_K^{1/2}}{4c}\ \ \forall t\in[T_0,T_{j+1}].
\]
\end{thm}
\begin{proof}
We restrict time to be in an arbitrary interval $[T_k,T_{k+1}]$ and take the inner product of (\ref{nse1g}) with Aw
\begin{align}\label{abc}
\frac{1}{2}\frac{d}{dt}\|w\|^2 + \nu|Aw|^2 + \mu_{k+1}\|w\|^2  &= (B(w,w),Aw) + (f,Aw) \nn \\
&+ \mu_{k+1}(Q_Kw,Aw) + \mu_{k+1}(P_K u, Aw).
\end{align}
Each term on the right hand side is estimated as before:
\begin{align*}
&|(B(w,w),Aw)|\le c\|w\|^{3/2}|Aw|^{3/2}\le \frac{\nu}{8}|Aw|^2 + \frac{c}{\nu^3}\|w\|^6\\
&|(f,Aw)|\le \frac{1}{\nu}|f|^2 + \frac{\nu}{4}|Aw|^2\\
& \mu_{k+1}|(Q_Kw,Aw)|=\mu_{k+1}\|Q_K w \|^2 \le \frac{ \mu_{k+1}}{\lambda_K} |Aw|^2\le \frac{\nu}{8}|Aw|^2\\
& \mu_{k+1}|(P_K u, Aw)|\le \frac{\mu_{k+1}}{2}\|P_K u\|^2 + \frac{\mu_{k+1}}{2}\| w \|^2.
\end{align*}
Inserting the estimates into (\ref{abc}),
\[
\frac{d}{dt}\| w \|^2 + \left(\mu_{k+1}-\frac{c}{\nu^3}\|w\|^4\right)\|w\|^2 \le \frac{2}{\nu}|f|^2 + \mu_{k+1}\|P_K u\|^2.
\]
Using the lower bound for $\mu_{k+1}$ in (\ref{adaptivemucond}) and defining $\tau=t-T_k$ we obtain by the Gronwall inequality $\forall t\in[T_k,T_{k+1}]$,
\[
\|w(t)\|^2 \le e^{-(\mu_{k+1}/2) \tau}\|w(T_j)\|^2 +
\left(1-  e^{-(\mu_{k+1}/2) \tau}\right)\left(\frac{4|f|^2}{\nu^2\lambda_1} + 2\widetilde{M}_{k+1}^2\right)\le M_{k+1}^2.
\]
\end{proof}
\begin{rem}In order for this iterative construction to work up to time interval $[T_j,T_{j+1}]$ we need to satisfy $M_k^4\lambda_K^{-1}\le \frac{\nu^4}{16c},\ k\in\{0,1,..., j+1\}$. By construction, $M_1^2=\frac{4|f|^2}{\nu^2\lambda_1} + 2\widetilde{M}_{1}^2$ and $M_{j+1}^2\le\max\bigg\{M_j^2, \left(\frac{4|f|^2}{\nu^2\lambda_1} + 2\widetilde{M}_{j+1}^2\right)\bigg\}$ which means such a choice of $\{\mu_1,\mu_2,...,\mu_{j+1}\}$ is possible if we assume (\ref{adaptiveassumption}).
\end{rem}

\begin{thm}
Suppose $I_h=P_K$ is a modal interpolant and satisfies (\ref{type1k}). Assume the same conditions from Theorem \ref{wboundg}.  Then the solutions of (\ref{nse1}) and (\ref{nse1g}) satisfy
\[
|\widetilde{w}(t)|^2\le e^{-(\mu_{k+1}/2)t}|\tw(T_k)|^2\ \forall t\in[T_k,T_{k+1}],\ \forall k\in\{0,1,...,j\}
\]
where $\widetilde{w}=w-u$.
\end{thm}

\begin{proof}
Assume $N\ge K$, $t\in[T_k,T_{k+1}]$, and $\widetilde{w}_N$ satisfies the following,
\begin{align*}
& \frac{d\tw_N}{dt} + vA\tw_N + P_N B(w_N,w_N)-P_NB(u_N,u_N)\\
&\qquad  = -\mu_{k+1} P_K(\tw_N) + \mu_{k+1} P_K(u-u_N)
\end{align*}
which can be rearranged to
\begin{align*}
&\frac{d\tw_N}{dt} + vA\tw_N + P_NB(\tw_N,w_N)+P_NB(u_N,\tw_N)\\
&\qquad = -\mu_{k+1} Q_K(\tw_N) -\mu_{k+1} \tw_N + \mu_{k+1} P_K(u-u_N).
\end{align*}
We take the inner product with $\tw_N$
\begin{align}\label{eqn8888}
& \frac{1}{2}\frac{d}{dt}|\tw_N|^2 + \nu\|\tw_N\|^2 + \mu_{k+1} |\tw_N|^2 \nn \\
&\qquad= -(B(\tw_N,w_N),\tw_N) -\mu_{k+1}(Q_K(\tw_N),\tw_N)\nn \\
&\qquad + \mu_{k+1} ( P_K(u-u_N),\tw_N).
\end{align}
We now  estimate each term on the right hand side as follows:
\begin{align*}
&|(B(\tw_N,w_N),\tw_N)|\le c|\tw_N|^{1/2} \|w_N\| \|\tw_N\|^{3/2}\\
&\qquad \le \frac{c}{\nu^3}\|w_N\|^4|\tw_N|^2 + \frac{\nu}{2}\|\tw_N\|^2\\
& \mu_{k+1}|(Q_K(\tw_N),\tw_N)|\le \mu_{k+1} |Q_K(\tw_N)||\tw_N| \\
&\qquad \le \frac{\mu_{k+1}}{\lambda_K^{1/2}}\|\tw_N\| |\tw_N| \le \frac{\mu_{k+1}}{\lambda_K}\|\tw_N\|^2 + \frac{\mu_{k+1}}{4} |\tw_N|^2 \\
&\qquad \le \frac{\nu}{2}\|\tw_N\|^2 + \frac{\mu_{k+1}}{4}|\tw_N|^2 \\
&\mu_{k+1} |(P_K(u-u_N),\tw_N)|\le \mu_{j+1} |P_K(u-u_N)||\tw_N| \\
&\qquad \le \mu_{k+1} |P_K(u-u_N)|^2 + \frac{\mu_{k+1}}{4}|\tw_N|^2.
\end{align*}
Inserting the estimates into \eqref{eqn8888}, 
\[
\frac{d}{dt}|\tw_N|^2 + (\mu_{k+1}-\frac{c}{\nu^3}\|w_N\|^4) |\tw_N|^2 \le \mu_{k+1} |P_K(u-u_N)|^2
\]
and since $\mu$ satifies (\ref{adaptivemucond}),
\[
\frac{d}{dt}|\tw_N|^2 + \frac{\mu_{k+1}}{2} |\tw_N|^2 \le \mu_{k+1} |P_K(u-u_N)|^2 .
\]
By Gronwall inequality,
\[
|\tw_N(t)|^2\le e^{-(\mu_{k+1}/2)t}|\tw(T_k)|^2 + 2\sup_{t\in[T_k,T_{k+1}]}|P_K(u-u_N)|^2.
\]
Recall $u_N\to u$ in $C(0,T;V')$ and $\phi_i$ is the ith eigenvector associated with A. Therefore,
\[
\liminf_{N\to\infty}|P_K(u-u_N)|\le \liminf_{N\to\infty}\sum\limits_{i=1}^{K} |(u-u_N,\phi_i)|=0.
\]
Since $\tw_N$ converges to $\tw$ weakly,
\[
|\tw(t)|^2\le e^{-(\mu_{k+1}/2)t}|\tw(T_k)|^2,\ \forall t\in[T_k,T_{k+1}]
\]
which proves the result.
\end{proof}



\bibliographystyle{amsplain}

\begin{thebibliography}{n} 


 \bibitem{AB}
 \newblock Albanez, Débora A. F.; Benvenutti, Maicon J. 
 \newblock Continuous data assimilation algorithm for simplified Bardina model. 
 \newblock \emph{Evol. Equ. Control Theory} {\bf 7} (2018), 
 no. 1, 33--52. 
 
  \bibitem{ANT}
  \newblock Albanez, Débora A. F.; Nussenzveig Lopes, Helena J.; Titi, Edriss S. 
  \newblock Continuous data assimilation for the three-dimensional Navier–Stokes-$\alpha$ model. 
  \newblock \emph{Asymptot. Anal.}  {\bf 97 } (2016), no. 1-2, 
  139--164. 
  
  \bibitem{AltafTitiKnioZhaoMcCabeHoteit2015} M.U. Altaf, E.S. Titi, T. Gebrael, O. Knio, L. Zhao, M.F. McCabe and I. Hoteit, Downscaling the 2D B\'enard convection equations using continuous data assimilation, \emph{Comput. Geosci.} {bf 21} (2017), no. 3, 393-410.

  
  
  \bibitem{asch} M. Asch, M. Bocquet and M. Nodet, \emph{Data Assimilation: Methods, Algorithms, and Applications}, Fundamentals of Algorithms, 11. Society for Industrial and Applied Mathematics (SIAM), Philadelphia, PA, 2016. 

\bibitem{AzouaniOlsonTiti2014} A. Azouani, E. Olson and E. S. Titi, \emph{Continuous Data Assimilation Using General Interpolant Observables}, J. Nonlinear Sci., 24 (2014), pp. 277--304.

\bibitem{AzouaniTiti2014} A. Azouani and E.S. Titi, \emph{Feedback control of nonlinear dissipative systems by finite determining parameters -- a reaction diffusion paradigm}, Evol. Equ. Control Theory, 3 (2014), no. 4, pp. 579--594.

\bibitem{BessaihOlsonTiti2015} H. Bessaih, E. Olson and E. S. Titi, \emph{Continuous data assimilation with stochastically noisy data}, Nonlinearity 28 (2015), pp. 729--753.



\bibitem{BFMT} A. Biswas, C. Foias, C. F. Mondaini and E. S. Titi, 
Downscaling data assimilation algorithm with applications to statistical solutions of the Navier–Stokes equations, \emph{Ann. Inst. H. Poincar\'{e} Anal. Non Lin\'{e}aire} {\bf 36} (2019), no. 2, 295–326. 

\bibitem{BiswasMartinez2017} A. Biswas and V.R. Martinez, \emph{Higher-order synchronization for a data assimilation algorithm for the 2D Navier-Stokes equations}, Nonlinear Anal. Real World Appl., 35 (2017), pp. 132--157.

\bibitem{BLSZ2013} D. Bl\"{o}mker, K. Law, A.M. Stuart and K.C. Zygalakis, \emph{Accuracy and stability of the continuous-time 3DVAR filter for the Navier-Stokes equation}, Nonlinearity 26 (2013), no. 8, pp. 2193--2219.

\bibitem{BLLMCSS2013} C.E.A. Brett, K.F. Lam, K.J.H. Law, D.S. McCormick, M.R. Scott and A.M. Stuart, \emph{Accuracy and stability of filters for dissipative PDEs}, Phys. D 245 (2013), pp. 34--45.

\bibitem{BV}
\newblock T. Buckmaster and V. Vicol.
\newblock  Nonuniqueness of weak solutions to the Navier-Stokes equation.
\newblock \emph{ Ann. of Math. (2)} {\bf 189} (2019), 
no. 1, 101--144.


\bibitem{CHL}
\newblock E. Carlson, J. Hudson and A. Larios.
\newblock Parameter recovery for the 2 dimensional Navier-Stokes equations via continuous data assimilation.
\newblock \emph{ SIAM J. Sci. Comput.} {\bf 42} (2020), 
no. 1, A250--A270.

\bibitem{cf} P. Constantin and C. Foias, \emph{Navier-Stokes Equations}, Chicago Lectures in Mathematics, University of Chicago Press, Chicago, IL, 1988.


\bibitem{chedai}
\newblock A. Cheskidov and M. Dai.
\newblock Determining modes for the surface quasi-geostrophic equation.
\newblock  \emph{Phys. D 376/377} (2018), 204--215.


\bibitem{Daley1991} R. Daley, \emph{Atmospheric Data Analysis}, Cambridge Atmospheric and Space Science Series, Cambridge University Press, Cambridge (1991).

\bibitem{desam2019} Srinivas Desamsetti; Hari Prasad Dasari; 
Sabique Langodan; Edriss S. Titi; Omar
Knio and Ibrahim Hoteit, Dynamical downscaling of general circulation models using continuous data assimilation, 
\emph{Quarterly Journal of the Royal Meteorological Society}, (2019), \emph{https://doi.org/10.1002/qj.3612}.

\bibitem{DiLeoni1} Di Leoni, Patricio Clark; Mazzino, Andrea and Biferale, Luca; Inferring flow parameters and turbulent configuration with physics-informed data assimilation and spectral nudging, \emph{Phys. Rev. Fluids} 
{\bf 3}, no. 10 (2018),104604. \emph{DOI:https://doi.org/10.1103/PhysRevFluids.3.104604} 

\bibitem{DiLeoni2} 
\newblock Di Leoni, Patricio Clark; Mazzino, Andrea and Biferale, Luca.
\newblock  Synchronization to big-data: Nudging the Navier-Stokes equations for data assimilation of turbulent flows.
\newblock \emph{Physical Review X} {\bf 10}, 011023 (2020),
011023-1--011023-15.



\bibitem{FGMW} 
\newblock Farhat, A.; Glatt-Holtz, N. E.; Martinez, V. R.; McQuarrie, S. A.; Whitehead, J. P.; 
\newblock Data Assimilation in Large Prandtl Rayleigh–Bénard Convection from Thermal Measurements.
\newblock \emph{ SIAM J. Appl. Dyn. Syst.} {\bf 19} (2020),
 no. 1, 510--540.
 
 
  \bibitem{FJJT}
  \newblock Farhat, Aseel; Johnston, Hans; Jolly, Michael; Titi, Edriss S. 
  \newblock Assimilation of nearly turbulent Rayleigh-B\'{e}nard flow through vorticity or local circulation measurements: a computational study. 
  \emph{J. Sci. Comput.} {\bf 77} (2018), no. 3, 1519--1533.
  
  \bibitem{FarhatJollyTiti2015} A. Farhat, M.S. Jolly, and E.S. Titi, \emph{Continuous data assimilation for the 2D B\'enard convection through velocity measurements alone}, Phys. D, 303 (2015), pp. 59--66.
  
  \bibitem{FLT} A. Farhat, E. Lunasin  and E.S. Titi, \emph{On the Charney conjecture of data assimilation employing temperature measurements alone: the paradigm of 3D planetary geostrophic model}, Mathematics of Climate and Weather Forecasting, 2(1) (2016), pp. 61--74.
  
    \bibitem{FLT1} Farhat, Aseel; Lunasin, Evelyn; Titi, Edriss S. \newblock Data assimilation algorithm for 3D B\'{e}nard convection in porous media employing only temperature measurements. 
    \newblock \emph{J. Math. Anal. Appl.} {\bf 438} (2016), 
    no. 1, 492--506.
 
  \bibitem{FLT2}
  \newblock Farhat, Aseel; Lunasin, Evelyn; Titi, Edriss S. 
  \newblock A data assimilation algorithm: the paradigm of the 3D Leray-$\alpha$ model of turbulence.
  \emph{Partial differential equations arising from physics and geometry,} 253--273, London Math. Soc. Lecture Note Ser., 450, \emph{Cambridge Univ. Press, Cambridge,} 2019. 




\bibitem{FoiasMondainiTiti2016} C. Foias, C. F. Mondaini and E. S. Titi, \emph{A discrete data assimilation scheme for the solutions of the two-dimensional Navier-Stokes equations and their statistics}, SIAM J. Appl. Dyn. Syst., 15 (2016), no. 4, pp. 2109--2142.






\bibitem{FoiasProdi1967} C. Foias and G. Prodi, \emph{Sur le comportement global des solutions non-stationnaires des \'equations de Navier-Stokes en dimension 2}, Rend. Sem. Mat. Univ. Padova, 39 (1967), pp. 1--34.

\bibitem{FoiasTemam1984} C. Foias and R. Temam, \emph{Determination of the solutions of the Navier-Stokes equations by a set of nodal values}, Math. Comp., 43 (1984), pp. 117--133.

\bibitem{FoTi91}
\newblock C. Foias and E. S. Titi,
\newblock Determining nodes, finite difference schemes and inertial manifolds.
\newblock  \emph{Nonlinearity} {\bf 4} (1991), no. 1, 135--153.

\bibitem{GeshoOlsonTiti2015} M. Gesho, E. Olson and E. Titi, \emph{A computational study of a data assimilation algorithm for the two dimensional Navier-Stokes equations}, Commun. Comput. Phys., 19 (2016), no. 4, pp. 1094--1110. 

\bibitem{HMbook2012} J. Harlim and A. Majda, \emph{Filtering Complex Turbulent Systems}, Cambridge University Press, Cambridge, 2012.

\bibitem{HOT} K. Hayden, E. Olson and E.S. Titi, \emph{Discrete data assimilation in the Lorenz and 2D Navier-Stokes equations,} Phys. D, 240 (2011), 1416-1425.

\bibitem{HJ2018} Joshua Hudson and Michael Jolly, Numerical efficacy study for data assimilation for the 2D magnetohydrodynamic equations, \emph{Journal of Computational Dynamics,} {\bf 6} no. 1 (2019), 131-145.


\bibitem{JT92}
\newblock D. A. Jones and E. S. Titi,
\newblock Determining finite volume elements for the 2D Navier-Stokes equations.
\newblock Experimental mathematics: computational issues in nonlinear science (Los Alamos, NM, 1991).
\newblock \emph{Physica D} (1992), no. 1-4, 165--174.

\bibitem{JT93}
\newblock D. A. Jones and E. S. Titi,
\newblock Upper bounds on the number of determining modes, nodes, and volume elements for the Navier-Stokes equations. 
\newblock  \emph{Indiana Univ. Math. J.} {\bf 42} (1993), no. 3, 875--887. 

\bibitem{Kalnay2003}
\newblock Kalnay, E.
\newblock \emph{Atmospheric Modeling, Data Assimilation and Predictability.}
\newblock Cambridge University Press, 2003.

\bibitem{KLS}
\newblock D. T. B. Kelly, K.J. H.  Law and A. M. Stuart,
\newblock Well-posedness and accuracy of the ensemble Kalman 
filter in discrete and continuous time.
\newblock \emph{Nonlinearity} (2014), 2579-2603.


\bibitem{larios2019} Larios, Adam; Rebholz, Leo G.; Zerfas, Camille. Global in time stability and accuracy of IMEX-FEM data assimilation schemes for Navier-Stokes equations, \emph{Comput. Methods Appl. Mech. Engrg.} {\bf 345} (2019), 1077--1093.




\bibitem{LSZbook2015}  K. Law, A.M. Stuart and K.C. Zygalakis, \emph{Data Assimilation, A Mathematical Introduction}, Springer, 2015.

\bibitem{LR}
     \newblock P.G. Lemari{\'e}-Rieusset, 
     \newblock \emph{Recent developments in the Navier-Stokes problem},
     \newblock Chapman \& Hall/CRC Research Notes in Mathematics. {\bf 431} (2002).
     
     \bibitem{MarkowichTitiTrabelsi2016} P.A. Markowich, E.S. Titi, and S. Trabelsi, \emph{Continuous data assimilation for the three-dimensional Brinkman-Forchheimer-extended Darcy model}, Nonlinearity, 29 (2016), pp. 1292--1328.
     

\bibitem{pei} 
\newblock Y. Pei.
\newblock Continuous data assimilation for the 3D primitive equations of the ocean. 
\newblock \emph{ Commun. Pure Appl. Anal.} {\bf 18} (2019), 
no. 2, 643--661. 


     
     

\bibitem{ReichCotterbook2015} S. Reich and C. Cotter, \emph{Probabilistic Forecasting and Bayesian Data Assimilation}, Cambridge University Press, Cambridge, 2015.

\bibitem{robinson} J. C. Robinson,
\emph{Infinite-Dimensional Dynamical Systems},
Cambridge Texts in Applied Mathematics. Cambridge University Press, Cambridge 2001. An introduction to dissipative parabolic PDEs and the theory of global attractors.

\bibitem{SS} J. Serrin, The initial value problem for the Navier-Stokes equations, 
\emph{Nonlinear Problems (R. E. Langer, ed.)}, University of Wisconsin Press, 1963, pp. 69-98.

\bibitem{Temam} R. Temam, \emph{Navier-Stokes Equations. Theory and Numerical Analysis}, Studies in Mathematics and its Applications, 3rd edition, North-Holland Publishing Co., Amsterdam-New York, 1984. Reedition in the AMS Chealsea Series, AMS, Providence, 2001.


\bibitem{tmk2016-1}  Xin T. Tong; Andrew J. Majda; David Kelly, Nonlinear stability of the ensemble Kalman filter with adaptive covariance inflation, \emph{Commun. Math. Sci.} {\bf 14} (2016), no. 5, 1283Ð1313. 

\bibitem{tmk2016-2} Xin T. Tong; Andrew J. Majda; David Kelly, Nonlinear stability and ergodicity of ensemble based Kalman filters, \emph{Nonlinearity} {\bf 29} (2016), no. 2, 657Ð691.


\bibitem{Rebholz2019} Zerfas, Camille, Rebholz, Leo G., Schneier, Michael, Iliescu, Traian; Continuous data assimilation reduced order models of fluid flow, \emph{Comput. Methods Appl. Mech. Engrg.} {\bf 357} (2019), 112596, 
18 pp. 






\end{thebibliography}

\end{document}